\documentclass[10pt]{amsart}
\usepackage{amsmath,amssymb,times,color}
\usepackage{graphicx}
\usepackage{dsfont}
\definecolor{webred}{rgb}{0.75,0,0}
\definecolor{webgreen}{rgb}{0,0.75,0}
\usepackage[citecolor=webgreen,colorlinks=true,linkcolor=webred]{hyperref}
\usepackage[english]{babel}
 \usepackage[color]{showkeys}
\definecolor{refkey}{gray}{0.75}

\numberwithin{equation}{section}

\date{\today}

\makeindex \makeatletter
\def\captionof#1#2{{\def\@captype{#1}#2}}
\makeatother

\newtheorem{lemma}{Lemma}[section]
\newtheorem{definition}{Definition}[section]
\newtheorem{proposition}{Proposition}[section]
\newtheorem{remark}{\bf Remark}
\newtheorem{theorem}{\bf Theorem}[section]
\def\bgneqnn{\begin{equation}}
\def\endeqnn{\end{equation}}
\def\bgneqy{\begin{eqnarray}}
\def\endeqy{\end{eqnarray}}
\def\bgneqy*{\begin{eqnarray*}}
\def\endeqy*{\end{eqnarray*}}

\def\text{\mbox}

\def\bgneqy*{\begin{eqnarray*}}
\def\endeqy*{\end{eqnarray*}}

\def\qed{\hfill$\blacksquare$\par\bigskip}

\newcounter{tablegroup}
\newcounter{subtable}[tablegroup]

\newcommand{\handletables}

\title{Analysis of a system modelling the interaction between the motion of piston-spring and a viscous gas }
\author[Chebbi, M\'acha, Ne\v casov\'a]{Sabrine Chebbi$^\spadesuit$ \and V\'aclav M\'acha$^\spadesuit$ \and \v S\'arka Ne\v casov\'a$^\spadesuit$}

\begin{document}

{ \abstract We are concerned with a one dimensional flow of a compressible fluid which may be seen as a simplification of the flow of fluid in a long thin pipe. We assume that the pipe is on one side ended by a spring. The other side of the pipe is let open -- there we assume either  inflow or outflow boundary conditions. Such situation can be understood as a toy model for human lungs. We tackle the question of uniqueness and existence of a strong solution for a system modelling the above process, special emphasis is laid upon the estimate of the maximal time of existence. }

\maketitle
\centerline{\scshape }

\begin{center}
$\spadesuit$: Institute of Mathematics of the Czech Academy of Sciences,\\ \v Zitn\'a 25, 11567 Praha 1\\
email:sabrine.chebbi@fst.utm.tn, macha@math.cas.cz, matus@math.cas.cz
\end{center}


%
%
%
%
%
%
%
\section{Introduction}
The 1D flow of a Newtonian fluid is a reasonable simplification of a complex problem because it allows us
to deduce the existence and uniqueness of smooth solution – this is a result which might be later used in many
applications from numerical simulations to control theory. This particular paper deals with boundary conditions
which might be seen as a toy model of human lungs – we assume there is a long thin pipe which is on one side
open and on the other side there is a piston on a spring and damper. The open part of the boundary is described
by either inflow boundary condition (inspiration) or outflow boundary condition (expiration). Here we would
like to mention that it would be interesting to take into account other kind of inflow-outflow boundary condition
– the most appropriate for our intentions would be a version of do-nothing conditions (a very nice paper about
such condition for incompressible flow is due to Bathory \cite{bathory} ). It is of interest that such boundary condition for
the compressible flow has not been stated yet.\\

The study of one dimensional compressible flows dates back to names such as Antontsev, Kazhikhov, Mohakhov, Ducomet et al., Shelukhin, Straskraba, Zlotnik (\cite{AnMoKa}, \cite{kazhikhov}, \cite{KaSh}) or Kawohl \cite{kawohl} on the western hemisphere.
The one dimensional fluid structure interaction (in particular a movement of a piston inside pipe) has been tackled
for example by Shelukhin in \cite{shelukhin} who considered thermally conducting piston inside a fluid, the same situation
with insulated piston was treated in \cite{FMNT} (for the isentropic case we refer to \cite{1}, the large time behavior is examined
in \cite{14}) and such piston inside a pipe with inflow boundary conditions was considered by Maity, Takahashi and
Tucsnak \cite{MTT}.\\

It is worthwhile to mention that the outflow boundary condition is one of the main novelty considered in this paper since, up to our knowledge, it has not been considered in one dimensional setting yet. The combination of the outflow boundary and spring is also unique and it brings questions of independent interest -- one of them is the maximal time of existence of the strong solution. Here we provide an estimate of the time of existence which comes from the energy inequality.\\

The model describing the flow of the isentropic compressible viscous gas in a one-dimensional domain is given as follows 
 \begin{equation}\label{Mod}
    \left \{ \begin{array}{lrl}
\rho_t(t,x)+(\rho(t,x)u(t,x))_x=0, \hspace{1.6cm} \ \text{for} \ t\geq 0, x\in (0,b(t)),  \ \ \  \text{\textit{(The continuity equation)}} \\
(\rho u)_t+ (\rho u^2)_x=(\mu u_x)_x - P_x(\rho), \hspace{1.2cm}  \text{for} \  t\geq 0, \ x\in (0,b(t)).\\
   \end{array}\right.
 \end{equation}
The unknown $u:[0,\infty)\times (0,b(t))\to \mathbb R$, $\rho:[0,\infty)\times (0,b(t))\to (0,\infty)$ represent the velocity and density of the fluid. We assume an isotropic flow, i.e. there is $\gamma>1$ such that \begin{equation}\label{pressure.law}P(\rho)=\rho^\gamma.\end{equation}
Further, the fluid is assumed to be Newtonian, i.e.,  the viscosity $\mu\in \mathbb R$ is constant. 

The system is considered on a moving domain -- the value $b(t)$ is the third (and last) unknown and it is governed by a second order ODE containing also the boundary conditions. Namely, the position of a piston, whose mass is $1$ and which is on a spring whose stiffness is $k>0$ and whose damping coefficient is  $l>0$, is governed by 
 \begin{equation}\label{M1}
  \ddot{b}(t) +k(b(t)-b_b)+l\dot{b}(t)= P(\rho(t,b(t))- \mu u_x(t,b(t)),
 \end{equation}
where $b_b \in \mathbb{R}$ is the equilibrium position of the spring. We remark that the above equations have  sense as far as $b(t)$ is a positive number otherwise the domain is empty. 
The velocity of the piston coincides with the gas velocity near the piston, \footnote{In the paper we consider the continuity of the gas velocity with the piston velocity through the boundary. It can be considered also just continuity  in the normal direction and add the Navier type of boundary condition.} i.e.,
      \begin{equation}\label{M2}
          u(t,b(t))=\dot{b}(t).
      \end{equation}
The behavior of the unknowns on the other side of the domain is described either by the inflow boundary condition
\begin{equation} \label{M4}
    u(t,0)= u_{in} (t), \ 
    \ \ \rho(t,0)=\rho_{in}(t),\ \mbox{for }t\in [0,T^{\star}),
\end{equation}
for some $T^{\star}>0$, or by the outflow boundary condition
\begin{equation}
u(t,0) = u_{out}(t)\ \mbox{for }t\in [T^{\star},T].
\end{equation}
 The system is endowed with initial conditions for the velocity, density and initial position of the piston 
           \begin{equation} \label{M3}
             \begin{array}{lrl}
(u(0,x),\rho(0,x), b(0))=(u_0(x), \rho_0(x), b_0).
       \end{array}
        \end{equation}
\begin{remark}\label{rem:1}
 We assume two types of boundary conditions, namely $$\Gamma_{in} =\{t\in \mathbb{R}_{+}, \  u(t,0)=u_{in} (t)> 0\},$$
 and 
 $$\Gamma_{out}=\{t\in \mathbb{R}_{+}, \  u(t,0)=u_{out} (t)\leq 0\} $$
 (the outer normal vector is equal to $-1$ in 1D ). These two cases are examined separately and each of them is of independent interest. The first case is treated similarly to \cite{MTT} (with certain modification). The second case deserves more attention and it can be understood as the main contribution of this paper -- the difficulty arises as $\rho_{out}$ can not be prescribed, rather than that, it is just a value of an unknown $\rho$ at one endpoint. In our setting, $\Gamma_{in} = [0,T*)$ and $\Gamma_{out} = [T,T^*]$ for some $0<T^*<T<\infty$.
    \end{remark}
\noindent To sum up, we treat  the following system
\begin{equation}\label{11}
                \left \{ \begin{array}{lrl}
\rho_t+(\rho u)_x=0, \hspace{6cm} \   \ (t,x)\in \Omega_t,  \ \ \  \\
(\rho u)_t+ (\rho u^2)_x=(\mu u_x)_x - P_x(\rho), \hspace{3.5cm}    (t, x)\in \Omega_t, \\
\ddot{b}(t)+l\dot{b}(t) +K(b(t)-b_b)= (P(\rho)-\mu u_x)(t,b(t)), \hspace{1cm} t\in (0,T),\\
u(t,b(t))=\dot{b}(t), \hspace{4.8cm} t\in (0,T) \\
u(t,0)=u_{in}(t)>0, \ \rho(t,0)=\rho_{in}(t)>0, \hspace{1cm} t\in [0,T^{\star}),\\
u(t,0)=u_{out}(t)\leq 0, \hspace{4.1cm} t\in [T^{\star},T],\\
u(0,x)=u_0(x), \ \rho(0,x)=\rho_0(x), \hspace{2.3cm} x\in (0,b(t)), \\
b(0)=b_0, \ \dot{b}(0)=b_1,\\
b(T^{\star})=b^{\star}, \dot{b}(T^{\star})=b_1^{\star},
       \end{array}\right.
        \end{equation}
such that  $$\Omega_t=\{(t,x)\in [0,T]\times (0,b(t))\}.$$

The rest of the article is organized as follows. 
In section \ref{s1}, we introduce the necessary variable changes that fix the moving domain and we decompose our system with respect to the boundary type condition (Inflow and Outflow). The notion of strong solution is also specified there. The inflow boundary condition is then treated in Section \ref{s2} where we state the local-in-time and global-in-time existence of the strong solution -- only a sketch of the proofs are provided due to similarities with \cite{MTT}. Finally, the section \ref{25} is devoted to the outflow boundary condition. Here we provide the local-in-time existence and we give an estimate on time for which there exists the strong solution. Namely, we claim that the solution exists till the contact between the piston and the free wall ($b(t) = 0$) and we give an estimate on time for which this situation does not appear.
     \section{Lagrangian coordinates} \label{s1}
     In this section, we introduce the Lagrangian mass change of coordinates to the system \eqref{11}. 
     It is worth to mention that the system in the Lagrangian mass coordinate is also on time dependent domain. This is the consequence of the inflow-outflow boundary condition. Therefore, the section is concluded with a transformation of Lagrangian coordinates to a time-independent domain.
     
     We denote by $\chi$ the new space coordinate given as follows  
     $$ y=\chi(t,x), \  \chi(t,x)=\int_{b(t)}^x \rho(t,s)\, {\rm d}s=-\int_{x}^{b(t)} \rho(t,s)\, {\rm d}s,  \  \  (t,x) \in \Omega_t.$$
     The variable $y$ ranges from $-\eta(t)$ to $0$ where $\eta(t)$ is defined as
     $$
     \eta(t)= \int_0^{b(t)}\rho(t,x)\, {\rm d}x.
     $$
     We note that 
\begin{multline*}
    \partial_t \eta(t)= -\frac{d}{dt} \left(\int^{b(t)}_0 \rho(t,y)\, {\rm d}y\right)= -\int^{b(t)}_0 \rho_t(t,y)\, {\rm d}y -\dot{b}(t) \rho(t,b(t)) \\  = -\int^{b(t)}_0 (\rho u)_x(t,y)\, {\rm d}y -\dot{b}(t) \rho(t,b(t))
    = \rho(t,b(t))u(t,b(t))-\dot{b}\rho(t,b(t))+\rho(t,0)u(t,0). \\ 
    \end{multline*}
Consequently
    \begin{equation}\label{eta.out}
 \partial_t \eta(t) = \rho(t,0)u(t,0)\neq 0
 \end{equation}
and $\eta $ is not a constant in time, therefore the Lagrangian domain $(-\eta(t),0)$ is still moving in time.
We distinguish two cases which will be treated separately (see Remark \ref{rem:1}):
     $$\eta_{in}(t)=\int_0^{b_0} \rho_0(x)\, {\rm d}x +\int_{T^{\star}}^t u_{in}(\tau) \rho_{in}(\tau)\, {\rm d}\tau, \ \  t\in [0,T^{\star}),$$
     and 
     \begin{equation*}\eta_{out}(t)=\int_{T^{\star}}^{b(T^{\star})} \rho(T^{\star},x)\, {\rm d}x +\int_{T^{\star}}^t u_{out}(\tau) \rho(\tau,0)\, {\rm d}\tau, \ \ t\in [T^{\star},T].\end{equation*}  
     \begin{remark}
        Note that  $\eta_{out}$ is a generic unknown as $\rho(s,0)$ is not a prior given and it is a part of solution.  
     \end{remark}

     We set
     $$\Tilde{\Omega}_{T^{\star}} = \{(t,x), \ 0<t<T^{\star}  \ and \ x\in (-\eta_{in}(t),0) \}, $$
     $$\Tilde{\Omega}_{T} = \{(t,x), \ T^{\star}<t<T \  and \ x\in (-\eta_{out}(t),0) \}.$$
     For each $t\in [0,T]$, we denote by $\chi^{-1}(t,.)$ the inverse map of $\chi(t,.)$.\\
The specific volume $v$ in mass Lagrangian coordinate is defined as an inverse of the density, namely \begin{equation*}
                \left \{ \begin{array}{lrl}
                v(t,y)=\frac{1}{\rho(t,\chi^{-1}(t,y))}, \hspace{0.5cm} t\in [0,T], \ y\in [-\eta(t),0]
                \\
                \rho(t,x)=\frac{1}{v(t,\chi(t,x))}, \hspace{0.7cm} t\in [0,T], \ x\in [0,b(t)].
       \end{array}\right.
        \end{equation*}
Similarly, the velocity field in Lagrangian mass coordinates reads 
    \begin{equation*}
        \left \{ \begin{array}{lrl}
        \tilde{u}(t,y)= u(t,\chi^{-1}(t,y)),   \hspace{0.5cm} t\in [0,T], \ y\in [-\eta(t),0]
        \\
         u(t,x)=\tilde{u}(t,\chi(t,x)), \hspace{0.9cm} t\in [0,T], \ x\in [0,b(t)].
   \end{array}\right.
    \end{equation*}        
Consequently, we have new unknowns $v$, $\tilde u$ and $b$ and we establish a new function for pressure (see \eqref{pressure.law})
\begin{equation*}
 q(v)(t,y):=P\left(\frac{1}{\rho(t,x)}\right) = v^{-\gamma}(t,y),
\end{equation*}
With the above notation the system \eqref{11} can be written with respect to the input flow in the form
 \begin{equation}\label{55}
    \left \{ \begin{array}{lrl}
v_t-\tilde{u}_{y}=0, \ \hspace{1.7cm} (t,x)\in \Tilde{\Omega}_{T^{\star}},\\
\tilde{u}_t+q_{y}(v)=\mu \left(\frac{\tilde{u}_{y}}{v}\right)_{y} \ \hspace{0.3cm} (t,x)\in \Tilde{\Omega}_{T^{\star}},\\
\ddot{b}(t)+l\dot{b}(t) +K(b(t)-b_b)= \left(q(v)-\mu \frac{\tilde{u}_{y}}{v}\right)(t,0), \hspace{1cm} t\in (0,T),\\
\tilde{u}(t,0)= \dot{b}(t), \ \tilde{u}(t,-\eta_{in}(t))=\tilde{u}_{in}(t) >0, \ \hspace{2.2cm} t\in (0,T^{\star}),\\
v(t,-\eta_{in}(t))=\frac{1}{\rho_{in}(t)}, \hspace{5.2cm} t\in (0,T^{\star}),\\
\tilde{u}(0,y)=\tilde{u}_0(y), \ v(0,y)=\frac{1}{\rho_0(x)}, \ \hspace{3.5cm} y\in (-\eta_{in}(t),0),\\
       \end{array}\right.
        \end{equation}
        and in the following form with respect to the output flow
      \begin{equation}\label{0101}
                \left \{ \begin{array}{lrl}
v_t-\tilde{u}_{y}=0, \ \hspace{1.7cm} (t,x)\in \Tilde{\Omega}_{T},\\
\tilde{u}_t+q_{y}(v)=\mu \left(\frac{\tilde{u}_{y}}{v}\right)_{y}, \ \hspace{0.3cm} (t,x)\in \Tilde{\Omega}_{T},\\
\ddot{b}(t)+l\dot{b}(t) +K(b(t)-b_b)= (q(v)-\mu \frac{\tilde{u}_{y}}{v})(t,0), \hspace{1cm} t\in (0,T),\\
\tilde{u}(t,0)= \dot{b}(t), \ \tilde{u}(t,-\eta_{out}(t))=\tilde{u}_{out}(t)<0, \ \hspace{2cm} t\in (T^{\star},T),\\
v(t,-\eta_{out}(t)), \ \  \ \  \text{unknown}, \hspace{4.5cm} t\in (T^{\star},T),\\
\tilde{u}(T^{\star},y)=\tilde{u}^{\star}(y),\ v(T^{\star},y)=v^{\star}(y),  \hspace{3cm} y\in (-\eta_{out}(t),0).\\
       \end{array}\right.
        \end{equation}   
We obtained a system in which the equation of the viscous gas holds in a domain which still depends on time -- this is a consequence of the inflow/outflow boundary conditions which do not preserve the total mass. We continue as in \cite{MTT}, and we introduce a second change variables in order  to rewrite the system \eqref{0101} in a fixed domain.
        To this aim we define
      \begin{equation}\label{cc}
            z= \Gamma(t,y)= \left \{ \begin{array}{lrl}           
                \frac{y}{-\eta_{in}(t)}, \ \ y\in [-\eta_{out},0] \  \text{and}  \ t\in [0,T^{\star}), \\
               \frac{y}{-\eta_{out}(t)}, \ \ y\in [-\eta_{out},0] \  \text{and}  \ t\in [T^{\star},T].
       \end{array}\right.
        \end{equation}
        It is easy to verify that, for every $t\in [T^{\star},T]$,  $\Gamma(t,.) $ is one to one from   $[-\eta_{out}(t),0]$ onto $[0,1]$. We set $z=\Gamma(t,y)$, the inverse of this map is given by 
         \begin{equation}\label{cc}
          y=  \Gamma^{-1}(t,z)=\left \{ \begin{array}{lrl}
            -\eta_{in}(t) z, \  z \in [0,1], \ t\in [0,T^{\star}), \\
            -\eta_{out}(t)z, \ \  \  z \in [0,1], \ t\in [T^{\star},T],
       \end{array}\right.
        \end{equation}
        The new specific volume $\tilde{v}$ and the velocity field with respect to the above coordinates have the form
         \begin{equation}\label{c}
                \left \{ \begin{array}{lrl}
                \tilde{v}(t,z)=v(t,\Gamma^{-1}(t,z)), \hspace{0.5cm} t\in [T^{\star},T], \ z\in [0,1],
                \\
                v(t,y)=\tilde{v}(t,\Gamma(t,z)), \hspace{0.7cm} t\in [T^{\star},T], \ y\in [-\eta_{out}(t),0].
       \end{array}\right.
        \end{equation}

        \begin{equation}\label{b}
                \left \{ \begin{array}{lrl}
                \b{u}(t,z)= \tilde{u}(t,\Gamma^{-1}(t,z)),   \hspace{0.5cm} t\in [T^{\star},T], \ z \in [0,1],
                \\
             \tilde{u}(t,y)=\b{u}(t,\Gamma(t,y)) \hspace{0.9cm} t\in [T^{\star},T], \ y\in [-\eta_{out}(t),0].
       \end{array}\right.
        \end{equation}
        With the above new notation the system \eqref{55} can be written in the form 
         \begin{equation}\label{0000}
                \left \{ \begin{array}{lrl}
\tilde{v}_t+\beta \tilde{v}_{z}-\alpha \b{u}_{z}=0, \ \hspace{2.6cm} (t,z)\in \underline{\Omega}_{T^{*}},\\
\b{u}_t+\beta \b{u}_{z} =\mu \alpha \left(\alpha \frac{\b{u}_{z}}{\tilde{v}}\right)_{z} -\alpha [q(\tilde{v})]_{z}, \ \hspace{0.4cm} (t,z)\in \underline{\Omega}_{T^{*}},\\
\ddot{b}(t)+l\dot{b}(t) +K(b(t)-b_b)= \left[ [q(\tilde{v})]-\mu \left(\alpha \frac{\b{u}_{z}}{\tilde{v}}\right) \right] (t,1) , \hspace{1cm} t\in (0,T^{\star}),\\
\b{u}(t,1)= \dot{b}(t), \ \b{u}(t,0)=\tilde{u}_{int}(t)>0, \ \hspace{3.8cm} t\in (0,T^{\star}),\\
\tilde{v}(t, 0), \  \text{prescribed} \hspace{1.5cm} t\in (0,T^{\star}),\\
\b{u}(0,z)=\b{u}_0(z),\ \hspace{2cm} z \in (0,1),\\
       \end{array}\right.
        \end{equation}
        and the system \eqref{0101} can be written in the form 
  \begin{equation}\label{000}
                \left \{ \begin{array}{lrl}
\tilde{v}_t+\beta \tilde{v}_{z}-\alpha \b{u}_{z}=0, \ \hspace{2.6cm} (t,z)\in \underline{\Omega}_{T},\\
\b{u}_t+\beta \b{u}_{z} =\mu \alpha \left(\alpha \frac{\b{u}_{z}}{\tilde{v}}\right)_{z} -\alpha [q(\tilde{v})]_{z}, \ \hspace{0.4cm} (t,z)\in \underline{\Omega}_{T},\\
\ddot{b}(t)+l\dot{b}(t) +K(b(t)-b_b)= \left[ [q(\tilde{v})]-\mu \left(\alpha \frac{\b{u}_{z}}{\tilde{v}}\right) \right] (t,1) , \hspace{1cm} t\in [T^{\star}, T),\\
\b{u}(t,1)= \dot{b}(t), \ \b{u}(t,0)=\tilde{u}_{out}(t)<0, \ \hspace{3.8cm} t\in [T^{\star},T),\\
\tilde{v}(t, 0)\ \  \ \  \text{unknown}, \hspace{1.5cm} t\in (T^{\star},T),\\
\b{u}(T^{\star},z)=\b{u}^{\star}(z),\ \hspace{2cm} z \in (0,1),\\
       \end{array}\right.
        \end{equation}          
where 
$$\underline{\Omega}_{T^{*}}=\{(t,z), \ \text{such that},\ 0<t<T^{\star}, \ 0<z <1\},$$
$$\underline{\Omega}_{T}=\{(t,z), \ \text{such that},\ T^{\star}<t<T, \ 0<z <1\},$$

and 
\begin{equation}\label{a}
                \left \{ \begin{array}{lrl}
             \alpha(t,z)  = \left\{
             \begin{array}{l}
             \frac{-1}{\eta_{in}(t)},\quad t\in [0,T^*)\\
             \frac{-1}{\eta_{out}(t)},\quad t\in [T^{\star},T]
             \end{array}\right.\ z \in [0,1],
                \\
          \beta(t,z) = \left\{\begin{array}{l}
          -z \frac{\dot{\eta}_{in}(t)}{\eta_{in}(t)},\quad t\in [0,T^*)\\
          - z \frac{\dot{\eta}_{out}(t)}{\eta_{out}(t)},\quad t\in [T^{\star},T], \end{array}\right.\ z \in [0,1].
       \end{array}\right.
        \end{equation}
        
        \begin{definition} \label{def0}
        The triple $$(\rho,u,b)\in C([0,T],L^\infty(0,b(t)))\times C([0,T],H^1(0,b(t)))\times H^2(0,T)$$
        fulfilling $\rho_t\in C([0,T],H^0(0,b(t)))$, $u_t\in C([0,T],H^0(0, b(t)))$, $u\in C([0,T],H^2(0,b(t))$ and $\frac 1\rho\in C([0,T]\times[0,b(t)])$ is a strong solution to \eqref{11}.
        \end{definition}
\begin{remark}\label{remark.sol} It is worthwhile to mention that the strong solution as defined above is smooth enough and it allows to switch arbitrarily between the formulations in different coefficients. Therefore, we are not going to provide the definition of strong solution for all formulation as we believe that the provided definition is sufficient.
\end{remark}
\section{Inflow boundary conditions}\label{s2}
The result covering the existence of strong solution on a time interval $[0,T^*]$ with $T^*$ given is (as usual) divided into two parts, where the first part yields the existence on the short time interval. This 
is given by the following theorem
\begin{theorem}\label{th1}
We assume that the variables of the system  $(u,\rho,b)$ satisfy the following assumptions for the initial and the boundary conditions
\begin{itemize}
    \item $b_b\in \mathbb R, b_1=\dot{b}(0) \in \mathbb{R},$
    \item 
$u_0(b_0)=u^b_0 \in \mathbb{R}$,  $u_0 \in H^1(0,1)$ \\ $u_{in} \in H^{1}([0,T^{\star}))$ and $u_{in}(t) >0 , \ t\in [0,T^{\star}).$  
\item
The initial density $\rho_0\in H^1(0,1)$ and $\rho_0(x)>0$  for every $x\in (0,1).$\\
 $\rho_{in}\in H^{1}([0,T^{\star}))$ and $ \rho_{in}(t)>0, \ t\in [0,T^{\star}).$
 $$\|\rho_{in}\|_{H^1([{0,T^{\star}}))} +\|u_{in}\|_{H^1({[0,T^{\star})})}\leq M,$$
$$u_{in}(t), \rho_{in}(t) \geq \frac{1}{M}, \  \  \text{for all}  \ t\in [0,T^{\star}).  $$
\item Moreover, we assume that there exists a constant $M>0$ such that 
\begin{align*}
    \|u_0\|_{H^{1}(0,1)} +\|\rho_0\|_{H^1(0,1)}+|u_{0}^b|\leq M, \\
   \frac{1}{M} \leq\rho_0(x)\leq M,\  \ x\in (0,1),\\
   \frac{1}{M}\leq b_0.
\end{align*}
\end{itemize}
Then, there exists $T_0\in (0,T^{\star})$, depending only on $M$, such that the system formed by \eqref{Mod},\eqref{M1}, \eqref{M2} with initial and boundary conditions \eqref{M4} and \eqref{M3} admits a unique strong solution in the regularity class specified in Definition \ref{def0} on $[0,T_0]$.
\end{theorem}
\begin{proof}
The proof of the above theorem is based on a "monolithic" linearization of the system of  \eqref{0000} and on an application of the Banach fixed point, the corresponding step consists in first solving uncoupled linear parabolic equation taking in to account the piston-spring motion, with non homogeneous boundary conditions. More precisely, we consider the following linear parabolic type system 
\begin{equation}\label{lin.in.sys}
                \left \{ \begin{array}{lrl}
\b{u}_t-\alpha_0\left(\frac{\alpha_0}{\tilde{v}_0} \b{u}_{z}\right) = f_1(t), \ \hspace{0.4cm} (t,z)\in [0,T^{\star})\times (0,1),\\
\b{u}(t,0)=\dot{b}(t), \ t\in  [0,T^{\star}),\\
\ddot{b}(t)+l\dot{b}(t) +K(b(t)-b_0)= \left[\frac{\alpha_0}{\tilde{v}_0} \b{u}_{z}\right](t,0)+f_2(t) , \hspace{1cm} t\in [0,T^{\star}),\\
\b{u}(0,z)=\b{u}_0(z),\ \hspace{2cm} z \in (0,1),\\
\b{u}(t,0)=\b{u}_{in}(t),\hspace{1cm} t\in [0,T^{\star}), \\
b(0)=b_0, \ \dot{b}(0)=b_1,
       \end{array}\right.
        \end{equation}   
where $f_1$ and $f_2$ are given sources terms. 
\begin{remark}
Due to the fact that the system \eqref{lin.in.sys} takes into account the spring which given is  by  the equation \eqref{lin.in.sys}$_3$ comparing to the linear parabolic system obtained in \cite{MTT} is the same studied system in \cite{spring} with $\alpha_0=\tilde{v}_0=1, \ \b{u}_{in}=0$, 
The proof of the result of  an existence and uniqueness local in time  is an adaptation of the proof of Theorem 4.1 from \cite{spring}.  So we omit the details. Moreover, we have the following estimate
\begin{equation}\label{est}
\|\b{u}_{z}\|_{L^2([0,T_0],L^{\infty}(0,1))} \leq KT_0^{1/8}.
\end{equation}
\end{remark}
The second step, is to use the velocity field $\b{u}$ solution of \eqref{lin.in.sys}, to solve the following  initial and boundary value problem 
\begin{equation}\label{3.3}
                \left \{ \begin{array}{lrl}
\tilde{v}_t+\beta \tilde{v}_{z}=\alpha \b{u}_{z}, \ \hspace{2cm} (t,z)\in [0,T^{\star})\times (0,1),\\
\b{u}(t,0)=\b{u}_{in}(t)>0, \ \hspace{1.5cm} t\in [0,T^{\star}),\\
\tilde{v}(t, 0)=\frac{1}{\rho_0(t)},\hspace{2.3cm} t\in (0,T^{\star}),\\
\Tilde{v}(0,z)=\tilde{v}_0(z),\ \hspace{2cm} y \in (0,1).\\
       \end{array}\right.
        \end{equation}
Using the estimation \eqref{est}, we deduce a local existence of strong solution for the transport equation with   unnecessary vanishing boundary conditions. For the proof we use appendix \cite{MTT} and corollary 2.3 \cite{H}.
\end{proof}

The solution obtained above might be extended up to an arbitrary time $T^*$ -- this is proven in the theorem following.
\begin{theorem}
Let $T^{\star}>0$ be an arbitrary time for the moment. Let us assume that $(\rho_0,u_0,b_0,b_1,u_{in},\rho_{in})$ satisfy the assumptions of Theorem \ref{th1}.
Then, the problem formed by \eqref{Mod}, \eqref{M1} with the boundary conditions \eqref{M2}, \eqref{M4} and the initial condition \eqref{M3} admits a unique strong solution on $[0,T^{\star})$.
\end{theorem}
\begin{proof}
It suffices to show that $b(t)$ remains bounded and positive for all $t\in [0,T^*]$ as then one may adopt the method from \cite{MTT}. To reach this goal, we use the energy inequality and lower bound of specific volume. Both is deduced in the  Lagrangian setting \eqref{55}. First, we introduce a function $w(t,y)\in C^1$ satisfying $w(t,-\eta_{in}) = -\dot \eta_{in}(t)$ and $w(t,0) = 0$. The transport theorem then yields:
\begin{multline}\label{ene.ine.in}
\partial_t \left(\frac 12 \int_{-\eta_{in}}^0 |\tilde u(t,y)|^2\, {\rm d}y\right) = \int_{-\eta_{in}}^0 \tilde u_t(t,y) \tilde u(t,y)\, {\rm d}y + \int_{\eta_{in}}^0\left(\frac 12 |\tilde u(t,y) |^2 w(t,y) \right)_y\, {\rm d}y\\
= \int_{-\eta_{in}}^0 \left(\mu \left(\frac{\tilde u_y}{v}\right) - q\right)_y\tilde u(t,y) \, {\rm d}y + \frac 12 |\tilde u(t,y)|^2w(t,0) - \frac 12 |\tilde u(t,y)|^2 w(t,-\eta_{in}(t))
\end{multline}
The last two terms on the right hand side are bounded due to the prescribed boundary conditions. For the first term, we compute as follows:
\begin{equation} \label{ene.ine.in}
\int_{-\eta_{in}}^0 \left(\mu \left(\frac{\tilde u_y}{v}\right) - q\right)_y\tilde u(t,y) \, {\rm d}y  = -\int_{-\eta_{in}}^0 \mu \left(\frac{\tilde u_y} v\right)\tilde u_y(t,y)\, {\rm d}y+ \int_{-\eta_{in}}^0 q \tilde u_y(t,y)\, {\rm d}y + \left[\left(\mu\left(\frac{\tilde u_y}v\right) - q \right)\tilde u(t,y)\right]_{-\eta_{in}}^0.
\end{equation}
The first term has just a correct sign and it will appear on the left hand side of the energy inequality. The second term is handled as follows
\begin{multline}
\int_{\eta_{in}}^0 q \tilde u_y(t,y)\, {\rm d}y = \int_{\eta_{in}}^0 q v_t(t,y) \, {\rm d}y  = \int_{\eta_{in}}^0 Q_t\, {\rm d}y = \partial_t \int_{-\eta_{in}}^0 Q\, {\rm d}y - \int_{-\eta_{in}}^0 \left(Qw(t,y)\right)_y\, {\rm d}y\\ = \partial_t \int_{-\eta_{in}}^0 Q\, {\rm d}y - Qw(t,0) + Qw(t,-\eta_{in})
\end{multline}
where $Q = Q(v) = -\int_v^\infty q(v)\, {\rm d}v$ is negative and the last two terms are bounded due to the prescribed boundary conditions.
It remains to handle the boundary term of \eqref{ene.ine.in}. In what follows, we use $\sigma$ as an abbreviation of the stress tensor, i.e., $\sigma =\mu \left(\frac{\tilde u_y}v\right)-q$. We have
$$\left[\sigma \tilde u(t,y)\right]_{-\eta_{in}}^0 = \sigma(t,0)\tilde u(t,0) - \sigma(t,-\eta_{in})u_{in}(t) = -\frac 12 \partial_t(\dot b^2(t)) - l\dot b(t)^2 - \frac 12K\partial_t\left(b(t) - b_b\right)^2 - \sigma(t,-\eta_{in}(t))u_{in}(t)$$
To handle the last term, we multiply the momentum equation of \eqref{55} by $u_{in}$ to deduce
$$
\sigma(t,-\eta_{in}) u_{in}(t) = -\partial_t \int_{-\eta_{in}}^0\tilde u_t(t,y) u_{in}(t)+ \tilde u (t,y) (u_{in}(t))_t\, {\rm d}y - u_{in}(t)^2 \dot \eta_{in}(t)
$$
We take all the calculations together and plug it into \eqref{ene.ine.in} to deduce

\begin{multline*}
\partial_t\left( \frac 12 \int_{-\eta_{in}}^0 |\tilde u(t,y)|^2\, {\rm d}y -
\int_{-\eta_{in}}^0 Q \, {\rm d}y + 
\frac 12\left(\dot b(t)^2 + K\left(b(t) - b_b\right)^2\right)\right)\\
+\int_{-\eta_{in}(t)}^0 \mu \left(\frac{\tilde u_y} v \right)\tilde u_y\, {\rm d}y + l \dot b(t)^2\\
=  Q w(t,-\eta_{in}) + \partial_t \int_{-\eta_{in}}^0 \tilde u_t(t,y) u_{in}(t)+ \tilde u(t,y)(u_{in}(t))_t\, {\rm d}y + u^2_{in}(t)\dot \eta_{in}(t) - \frac 12 |\tilde u|^2 w(t,-\eta_{in}(t))
\end{multline*}
We integrate over a time interval $[0,t]\subset [0,T^*]$ and we use the Gronwall inequality to get the energy estimate
\begin{multline*}
\frac 12 \int_{-\eta_{in}}^0 |\tilde u(t,y)|^2\, {\rm d}y - \int_{-\eta_{in}}^0 Q\, {\rm d}y + \frac 12\left(\dot b(t)^2 + K\left(b(t) - b_b\right)^2\right)
\\ +\int_0^{T^*}\int_{-\eta_{in}}^0 \mu \left(\frac{\tilde u_y} v \right)\tilde u_y\, {\rm d}y + l \dot b(t)^2\, {\rm d}t\leq C,
\end{multline*}
where $C$ depends on initial and boundary conditions and on $T^*$.\\
We take $M(v) = \mu \log v$, i.e., $M'(v) = \frac \mu v$. The momentum equation yields
\begin{equation}\label{int.vol}\tilde u_t = M_{ty} - q_y.\end{equation}
We integrate \eqref{int.vol} with respect to time and space over a set $[0,t]\times [y,0]$ where $t\in[0,T^*]$ and $y\in[-\eta_{in}(t),0]$ are chosen arbitrary. We arrive to 
$$\int_y^0 \tilde u(t,s)-\tilde u(0,s)\, {\rm d}x = M(t,0) - M(0,0) - M(t,y) +M(0,y) -  \int_0^tq(s,0) - q(s,y)\, {\rm d}s$$
The boundary condition yields 
$$\int_0^t q(s,0) = M(t,0) - M(0,0) + \dot b(t) + l b(t) + \int_0^t K(b(s) - b_0)\, {\rm d}s$$
and, consequently, 
$$
-M(t,y) = \int_y^0 u(t,s) - u(0,s)\, {\rm d}s + \dot b(t) + l b(t) + \int_0^tK(b(s) - b_0)\, {\rm d}s - \int_0^tq(s,y)\, {\rm d}s.
$$
The right hand side of the above inequality is bounded in terms of energy (recall $q$ is always positive) and thus
$$
-M(t,y)\leq C\ \Rightarrow \ \log \frac 1{v(t,y)}\leq C
$$
for  every $(t,y)\in [0,T^*]\times [-\eta_{in}(t),0]$ and, therefore,
$$
v(t,y)\geq \underline v \in \mathbb R^+
$$
The method of substitution yields
$$
\int_{\eta_{in}}^0 v(t,y)\, {\rm d}y = b(t),
$$
on the other hand,
$$
b(t) = \int_{\eta_{in}}^0 v(t,y)\, {\rm d}y\geq \eta_{in} \underline v > 0
$$
since $\eta_{in}$ is increasing and 
$\eta_{in}(0) >0$.
\end{proof}

\section{Outflow boundary conditions} \label{25}
%
%

\subsection{Local-in-time existence}
\begin{theorem}\label{outflow.local}
We assume the following initial conditions for the outflow case
\begin{equation}\label{ini.outflow}
\underline u(T^*,\cdot) \in H^1(0,1),\ \tilde v(T^*,\cdot)\in H^1(0,1),\ \frac 1{\tilde v(T^*,\cdot)}\in L^\infty (0,1),\ b(0)>0.
\end{equation}
Then,  there exists a strong solution $(\tilde v,\underline u,b)$ to \eqref{000} (in the sense of Definition \ref{def0}) on a time interval $[T^{\star},T_1]$ for certain $T_1>T^*$.
\end{theorem}
\noindent This result is achieved by a fixed point theorem. The proof is performed in the rest of this subsection and we first show that certain mapping is a contradiction. This is the content of the following proposition.
    \begin{proposition}\label{fixpunkt}
 We denote by $\Psi$ the following set of functions
        \begin{equation}
            \Psi=\{ g\in C^1([T^{\star},T]), \  \ g(T^{\star})=\eta_{0} \ \text{and} \ -m\leq \dot{g} \leq 0\} 
        \end{equation}
where $m$ is a constant greater than $\eta_0$. 
We define the operator $S$ as follows
        \begin{align}
            S(\eta_{out}(t))= \int_{T^{\star}}^{b(T^{\star})} \rho (T^{\star},x)\, {\rm d}x +\int_{T^{\star}}^t u_{out}(t) \rho(\tau,0)\, {\rm d}\tau \nonumber\\
            = \int_{T^{\star}}^{b(T^{\star})} \rho (T^{\star},x)\, {\rm d}x +\int_{T^{\star}}^t u_{out}(t) \frac{1}{\tilde{v}(\tau,0)}\, {\rm d}\tau 
        \end{align}
where $\tilde v$ is a solution to \eqref{000} with given $\eta_{out}(t)$       
   Then, we have \begin{itemize}
        \item  $S(\Psi) \subset \Psi$
        \item The mapping  $S:\Psi \to \Psi$ is a contraction in $C^1$ norm. 
    \end{itemize}
    \end{proposition}
\begin{proof}
%
%
We note here that $T$ can be chosen so close to $T^*$ that $g(t)\geq \frac 12 \eta_0$ for every $t\in [T^*,T]$. Next, the solution $(\underline u,\tilde v,b)$ to \eqref{000} exists -- the proof of the existence follows the lines of \cite[Section 3]{MTT} -- just a minor modification is necessary in order to accommodate our setting. Nevertheless, this particular existence result is written in Section \ref{con.rem} for reader's convenience. Furthermore, one can deduce that there exists $T_1>T^*$ such that there exists a solution $(\underline u,\tilde v,b)$ to \eqref{000} for every $\eta_{out}\in \Psi$.
\\ We also assume the time is so short that $b>0$ on $[T^{\star},T_1]$. In particular, 
\begin{equation*}
\begin{split}
\tilde v\in& C([T^{\star},T_1],H^1(0,1))\cap H^1((T^{\star},T_1),L^2(0,1))\\
\underline u\in& C([T^{\star},T_1],H^1(0,1))\cap L^2([T^{\star},T_1], H^2(0,1))\cap H^1([T^{\star},T_1],L^2(0,1)),\\
b\in& H^2(0,T_1),
\end{split}
\end{equation*}
uniformly with respect to $\eta_{out}$ and, moreover, there is $c>0$ such that
$$
\frac {1}{c}\leq \tilde v(t,x)\leq c
$$
for every $(t,x)\in (T^{\star},T_1)\times (0,1)$.

Next, using the fact that $u_{out}$ is a given continuous function, $T_1>T^{\star}$ is such that
    \begin{equation}
        \forall t\in [T^{\star}, T_1], \ |u_{out}(t)-u_{out}(T^{\star})|\leq \kappa,
    \end{equation}
where $\kappa>0$ is sufficiently small and it will be determined later.    
We multiply \eqref{000}$_1$ by $\tilde v$ and we use the integration by parts in order to deduce
$$
\partial_t\int_0^1 \frac 12 |\tilde v(t,z)|^2\, {\rm d}z - \frac 12\int_0^1 \beta_z(t,z) |\tilde v(t,z)|^2\, {\rm d}z + \frac 12\beta(t,1)\tilde v^2(t,1) = \int_0^1 \alpha(t,z) \underline u_z (t,z)\tilde v(t,z).
$$
Next, we differentiate \eqref{000}$_1$ with respect to $z$ and we multiply it by $\tilde v_z$ to get
$$
\frac 12 \int_0^1(\tilde v_z^2(t,z))_t\, {\rm d}z + \frac 12\int_0^1\beta_z(t,z) (\tilde v_z^2(t,z))\, {\rm d}z + \frac 12 \beta(t,1)(v_z(t,1))^2 = \int_0^1 \alpha(t,z) \underline u_{zz}(t,z) v_z (t,z)\, {\rm d}z.
$$
We sum these two equations and we skip several terms (recall $\beta(t,1)>0$ and $\underline u\in L^2((T^*,T_1), H^2(0,1))$) and we infer
$$
\partial_t \|\tilde v\|_{H^1}^2 \leq C\left( \|\tilde v\|_{H^1}^2 + 1\right).
$$
The Gr\"onwall lemma then yields
$$
\|\tilde v\|_{L^\infty(H^1)}^2 \leq C.
$$
Consequently, the equation \eqref{000}$_1$ itself then yields
$$
\|\tilde v_t\|_{L^\infty(L^2)} \leq C.
$$
By Gagliardo-Nirenberg 
\begin{multline*}
\|\tilde v(t,\cdot) - \tilde v(T^*,\cdot)\|_\infty \leq C\|\tilde v(t,\cdot) - \tilde v(T^*,\cdot)\|_{L^\infty(H^1)}^{3/4} \|\tilde v(t,\cdot) - \tilde v(T^*,\cdot)\|_{L^\infty(L^2)}^{1/4}\\
\leq C(T-T^*)^{1/4} \|\tilde v(t,\cdot) - \tilde v(T^*,\cdot)\|_{L^\infty(H^1)}^{3/4}\|\tilde v\|_{L^\infty(L^2)}^{1/4} \leq C(T-T^*)^{1/4}.
\end{multline*}
Thus, we have 
$$
S(\eta_{out})(T^*) = \int_{T^*}^{b(T^*)} \rho(T^*,x)\, {\rm d}x = \eta_0, \quad  S'(\eta_{out})(t) = \frac{u_{out}(t)}{\tilde v(\tau,0)} \leq 0,
$$
and
\begin{multline*}
\left|S'(\eta_{out})(t)\right| = \left|\frac{u_{out}(t)}{\tilde v(t,0)} - \frac{u_{out}(T^*)}{\tilde v(T^*,0)}  + \eta_0\right|\\
= \left|(u_{out}(t) - u_{out}(T^*))\frac 1{\tilde v(t,0)} +  u_{out}(0)\frac1{\tilde v(t,0)\tilde v(T^*,0)} (\tilde v(T^*,0) - \tilde v(t,0)) + \eta_0\right|,\\
\leq C\left(\kappa + (T_1-T^*)^{1/4}\right) + \eta_0,
\end{multline*}
and proper choice of $T_1>T^*$ (and consequently also the choice of $\kappa$) together with the assumption on $m$ yields $S:\Psi\to \Psi$.

Now, for the second result of the proposition, let $\eta_{out}^1,\eta_{out}^2\in \Psi $ and let 
$$
\alpha_i,\, \beta_i,\, \tilde v^i,\ \underline u^i,\ b_i,\ i\in \{1,2\}
$$
be appropriate constants and solutions to \eqref{000}.
We denote $\tilde{V} =\tilde{v}^1-\tilde{v}^2$, $\underline{U}=\underline{u}^1-\underline{u}^2$, and $B_i = b_1 - b_2$.
The system satisfied by $\tilde{V}$ has the form 
 \begin{equation}\label{ss0}
   \begin{split}
  \tilde{V}_t+\beta_1 \tilde{V}_{z}&= g(t,z)\qquad  (t,z)\in [T^{\star},T_1]\times [0,1],\\
  \tilde{V}(0,z)&=0\qquad z \in [0,1],
 \end{split}
\end{equation} 
where 
 \begin{equation}\label{gg}
  g(t,z)= \alpha_1 \underline{U}_{z} + (\alpha_1-\alpha_2) \b{u}^2_{z} + (\beta_2-\beta_1) \tilde{v}^1_{z}, \ \  \forall \ (t,z) \in [T^{\star},T_1] \times [0,1].
 \end{equation}
We proceed similarly to the previous paragraph -- we multiply \eqref{ss0}$_1$ by $\tilde V$, then we differentiate \eqref{ss0}$_1$ with respect to $z$ and we multiply it by $\tilde V_z$ in order to deduce the following:
\begin{equation}\label{eq:gronwall}
\partial_t \|\tilde V\|_{H^1} \leq C\left(\|\tilde V\|_{H^1} + \|g\|_{H^1}\right)
\end{equation}
The equation for $\underline U$ in a weak form read as follows
\begin{equation}\label{eq:gronwall2}
\underline U_t + \beta_1 \underline U_z - \mu \alpha_0\left(\alpha_0\frac{\underline U_z}{\tilde v_0}\right)_z = F
\end{equation}
where 
\begin{multline*}
F = \mu\left[\alpha_1\left(\alpha_1\frac{\underline u_z^1}{\tilde v^1}\right)_z - \alpha_0\left(\alpha_0\frac{\underline u_z^1}{\tilde v_0}\right) - \alpha_2\left(\alpha_2\frac{\underline u_z^2}{\tilde v^2}\right)_z +  \alpha_0\left(\alpha_0\frac{\underline u_z^2}{\tilde v_0}\right)\right]\\ + \left(\alpha_2-\alpha_1\right)q(\tilde v_1)_z  + \alpha_2 \left(q(\tilde v_1)_z - q(\tilde v_2)_z\right)
\end{multline*}
The unknown $\underline U$ is supposed to satisfy the following boundary conditions 
\begin{equation*}
\begin{split}
\underline U(t,1) &= 0,\ \underline U(t,0) =\dot B(t)\\
\ddot B(t)+ l\dot B(t) + K B(t)& = -\mu \left(\alpha_0 \frac{\underline U_z}{\tilde v_0}\right)(t,0) + f.
\end{split}
\end{equation*}
where 
$$
f = \alpha_1 q(\tilde v^1) - \alpha_2 q(\tilde v^2) - \mu \alpha_1\left(\alpha_1\frac{\underline u_z^1}{\tilde v^1}\right) + \mu \alpha_2\left(\alpha_2 \frac{\underline u_z^2}{\tilde v^2}\right) + \mu \alpha_0\left(\alpha_0\frac{\underline u_z^1}{\tilde v^1}\right) - \mu\alpha_0\left(\alpha_0\frac{\underline u_z^1}{\tilde v^2}\right)
$$
We multiply \eqref{eq:gronwall2} by $\underline U$ and integrate in order to get
\begin{equation}\label{eq:ene.par}
\partial_t \|\underline U\|_{L^2}^2 + \frac{\partial}{\partial t}\left(|\dot B|^2 + K|B|^2\right) + |\dot B|^2 + \|\underline U\|_{H^1}^2\leq C\left(\|\underline U\|_{L^2}^2 + \int_0^1 F U\, {\rm d}z + |f||\dot B|\right)
\end{equation}
Next, we multiply \eqref{eq:gronwall2} by $-\underline U_{zz}$ and we integrate the first term by parts. Next, the third term can be written as
$$
\mu\alpha_0^2\int_0^1 \frac 1{\tilde v_0}||\underline U_{zz}|\, {\rm d}z - \mu \alpha_0^2 \int_0^1\frac 1{\tilde v_0^2}\tilde v_{0z} \underline U_{z}\underline U_{zz}\, {\rm d}z
$$
where the second term can be estimated by means of the Gagliardo-Nirenberg and the regularity of $v_0$ inequality as follows
\begin{multline*}
\left|- \mu \alpha_0^2 \int_0^1\frac 1{\tilde v_0^2}\tilde v_{0z} \underline U_{z}\underline U_{zz}\, {\rm d}z\right|\leq C \|\underline U_z\|_{L^\infty}^2\|\tilde v_0\|_{H^1}^2 + \delta \|\underline U_{zz}\|_{L^2}^2\leq C\|\underline U_z\|^2_{L^\infty} + \delta \|\underline U_{zz}\|_{L^2}^2
\\
\leq C\|\underline U_{zz}\|_{L^2}\|\underline U_{z}\|_{L^2} + \delta \|\underline U_{zz}\|_{L^2}^2 \leq C\|\underline U_{z}\|^2_{L^2} + 2\delta \|\underline U_{zz}\|_{L^2}^2
\end{multline*}
for any $\delta>0$ sufficiently small.
In total, we deduce that
\begin{equation}\label{eq:ene.par2}
\partial_t\|\underline U\|_{H^1} + |\ddot B|^2 + \frac{\partial}{\partial t}|\dot B| + \|\underline U\|_{H^2}\leq C\left(\|\underline U\|_{H^1}^2 + \int_0^1 FU_{zz}\, {\rm d}z + |B||\ddot B|  + |f||\ddot B|\right)
\end{equation}
The definition of $\alpha_i$ and $\beta_i$ yields
\begin{equation*}
\|\alpha_1 - \alpha_2\|_{L^\infty} + \|\beta_2-\beta_1\|_{L^\infty([T^*,T]\times[0,1])} + \|(\beta_2)_z-(\beta_1)_z\|_{L^\infty([T^*,T]\times[0,1])} \leq C \|\eta_{out}^1 - \eta_{out}^2\|_{C^1}
\end{equation*}
and, moreover, the first bracket of $F$ can be written as
$$
\mu \left(\alpha_0^2 \left(\overline U_z (\tilde v^1 - \tilde v_0)\right)_z + \alpha_0^2\left(\overline u_z^2\tilde V\right)_z + \left(\alpha_1^2 - \alpha_2^2\right)\left(\frac{\overline u_z^2}{\tilde v^2}\right)_z + (\alpha_1^2 - \alpha_0^2)\left(\overline U_z  \frac 1{\tilde v^1} + \overline u_z^2 \frac 1{\tilde v^1\tilde v^2}\tilde V\right)_z \right)
$$
and $f$ can be seen as
$$(\alpha_1-\alpha_2)q(\tilde v^1) + \alpha_2\left(q(\tilde v^1) - q(\tilde v^2)\right) + \frac{\mu}{\tilde v^1}(\alpha_0^2 - \alpha_1^2)\underline U_z + \frac{\mu\underline u_z^2}{\tilde v^1\tilde v^2}(\alpha_0^2 - \alpha_1^2)\tilde V + \mu\frac{\underline u_z^2}{\tilde v^2}\left(\alpha_2^2 - \alpha_1^2\right)$$
and thus we may deduce
$$
|f|\leq C\left((\alpha_0^2 - \alpha_1^2)\|\underline U\|_{H^2} + (\alpha_0^2 - \alpha_1^2)\|\tilde V\|_{H^1} + (\alpha_2^2 - \alpha_1^2)\right)
$$
Due to the smoothness of coefficients and solutions $T_1>T^*$ can be chosen such that $|\alpha_1 - \alpha_0|$ and $|\tilde v_1 - \tilde v_0|$ are sufficiently small and we incorporate these into \eqref{eq:ene.par} and \eqref{eq:ene.par2} to deduce
\begin{equation}\label{eq:ene.par3}
\partial_t \|\underline U\|_{H^1}^2 + \|\underline U\|_{H^2}^2\leq C\left(\|\underline U\|_{H^1}^2+ \|\tilde V\|_{H^1}^2 +\|\eta_{out}^1 - \eta_{out}^2\|^2_{C^1}\right)
\end{equation}
The equations \eqref{eq:gronwall} and \eqref{eq:ene.par2} combined read as
\begin{equation*}
\partial_t\left(\|\tilde V\|_{H^1} + \|\underline U\|_{H^1}^2\right) + \|\underline U\|_{H^2}^2\leq C\left(\|\underline U\|_{H^1}^2 + \|\tilde V\|_{H^1}^2 +\|\eta_{out}^1 - \eta_{out}^2\|^2_{C^1}\right)
\end{equation*}
It is just a matter of routine 
to deduce
$$
\|\tilde U\|_{H^1} + \|\tilde V\|_{H^1}\leq C\|\eta_{out}^1 - \eta_{out}^2\|_{C^1}.
$$
Consequently, \eqref{ss0} yields
$$
\|\tilde V_t\|_{L^2}\leq C \|\eta_{out}^1 - \eta_{out}^2\|_{C^1}.
$$
Using the similar arguments as in the previous paragraph, we deduce 
$$
\|\tilde V(t,\cdot)\|_{L^\infty}\leq C(T-T^*)^{1/4}\|\eta_{out}^1 - \eta_{out}^2\|_{C^1}
$$
Consequently,
\begin{multline*}
   \left | S(\eta_{out}^1(t))-S(\eta_{out}^2(t)) \right | \leq \int_{T^{\star}}^t \left|u_{out}(s)\frac{1}{\tilde{v}^2(s,0)}-\frac{1}{\tilde{v}^1(s,0) }\right|\, {\rm d}s \\ 
   \leq C \int_{T^*}^{T} (T-T^*)^{1/4}  \, {\rm d}s \|\eta_{out}^1 - \eta_{out}^2\|_{C^1}, 
\end{multline*}
and 
\begin{equation*}
\left | \partial_t\left(S(\eta_{out}^1(t))-S(\eta_{out}^2(t))\right) \right | \leq \left|u_{out}(t)\frac{1}{\tilde{v}^2(t,0)}-\frac{1}{\tilde{v}^1(t,0) }\right| \leq C(T_1-T^*)^{1/4}\|\eta_{out}^1 - \eta_{out}^2\|_{C^1}
\end{equation*}
The proper choice of $T_1>T^*$ leads to the second claim of the lemma.
\end{proof}
Note that $\Psi\subset C^1([T^*,T_1])$ is a closed set and thus we apply the Banach fix point theorem in order to deduce the existence result of Theorem \ref{outflow.local}.
\subsection{Mass and energy estimates conservation}
We deduce the mass conservation law below. Although the presented form is not needed for our forthcoming computation, we are presenting it here because of politeness.
\begin{lemma}
            For every $t\in [T^{\star}, T_1]$ we have 
            \begin{equation}
                \int_0^{b(t)} \rho(t,x)\, {\rm d}x \leq C,
            \end{equation}
        \end{lemma}
        \begin{proof}
        We integrate the continuity equation in the system \eqref{11} on $[0,b(t)]$, similarly we obtain
        \begin{equation*}
            \frac{d}{dt} \left(\int_0^{b(t)} \rho(t,x)\, {\rm d}x \right)= \rho(t,0) u_{out}(t) <0, \ t\in (T^{\star}, T_1).
        \end{equation*}
        then, $\int_0^{b(t)} \rho(t,x)\, {\rm d}x $ is a 
 decreasing function, which mean that for all $t\in [T^{\star},T_1]$, we have 
 \begin{align*}
     \int_0^{b(t)} \rho(t,x)\, {\rm d}x \leq \int_0^{b_b} \rho(T^{\star},x)\, {\rm d}x.
 \end{align*}
 using the initial condition on $\rho(T^{\star},x)$, we obtain the desired estimate.
\end{proof}
Our next aim is to show estimates independent on time which allows to extend the solution beyond the time $T_1$ stated in Theorem \ref{outflow.local}. In what follows, we assume the strong solution exists on a nonempty time interval $[T^*, T_2]$ and we aim to deduce estimates which are independent on $T_2$ as far as $T_2$ is bounded.
We work with the system in Lagrangian coordinates for the rest of this section. In order to calculate the integrals properly, we use a letter $w$ for a function on $[T^*, T_2]\times \mathbb R$ which describes the movement of the domain, namely, it satisfies $w(t,-\eta_{out}(t)) = -\dot\eta_{out}(t)$ and $w(t,0) = 0$. It is used whenever we apply the Reynolds transport theorem (see \cite[Theorem 1.22]{NoSt} for reference). Recall, that the transport theorem yields
$$
\partial_t\int_{-\eta(t)}^0 f(t,y)\, {\rm d}y = \int_{-\eta(t)}^0 f_t(t,y)\, {\rm d}y + \int_{-\eta(t)}^0 (fw)_y(t,y)\, {\rm d}y
$$
for an arbitrary $C^1$ function $f:[T^*,T_0]\times [-\eta(t),0]\to \mathbb R$.\\
We would also like to recall that the method of substitution yields
\begin{equation}\label{ini.volume}
\int_{-\eta_{out}}^0 v(T^*,y)\, {\rm d}y = \int_{-\eta_{out}}^0 \frac1{\rho(T^*,y)}\, {\rm d}y = \int_0^{b(T^*)} 1\, {\rm d}x = b(T^*)
\end{equation}
and this is true for every $t\in [T^*,T_2]$, namely, we state the following lemma.
\begin{lemma}\label{specific.vol}
It holds that
$$
\int_{-\eta_{out}}^0 v(t,y)\, {\rm d}y =  b(t)
$$
for every $t\in [T^*,T_1]$.
\end{lemma}
 \begin{lemma}
       There exists a strictly positive constant $C$ depending on the initial condition and on $u_{out}$ such that the strong solution $(\rho,u,b)$ satisfies the following energy inequality for almost $t\in [T^{\star},T_1]$ assuming $T_1$ is lower than some given $T_1>T^*$. 
       \begin{multline} \label{EE}
           \int_{-\eta_{out}(t)}^0 \left( \frac{\tilde{u}^2}{2}-Q(v) \right)\, {\rm d}y +\frac{|\dot{b}(t)|^2}{2} +\frac{K}{2} |b(t)-b_b|^2 \\+  \int_{T^*}^t\int_{-\eta_{out}(t)}^0 \left|\frac{\mu}{v(s,y)} \tilde{u}^2_y(t,y) \right|dy +\int_{T^*}^t\frac{-\tilde{u}_{out}(s)}{\gamma-1} \frac{1}{v^{\gamma}(s,0)}\, {\rm d}s\leq C,
       \end{multline}
where $Q(v) = -\int_v^\infty q(z)\, {\rm d}z = \frac 1{1-\gamma}\frac 1{v^{\gamma-1}}$.
\end{lemma}
Let us recall that $\tilde u_{out}$ and $Q(s)$ are negative by definition and, therefore, all the terms on the right hand side of \eqref{EE} are positive.
       \begin{proof}
We rewrite the system as
        \begin{equation}\label{ss}
                \left \{ \begin{array}{lrl}
                v_t=\tilde{u}_y \\
                \tilde{u}_t=\sigma_y  
       \end{array}\right.
        \end{equation} 
where
        \begin{equation}
            \sigma(t,y)= \frac{\mu}{v}\tilde{u}_y -p\left(\frac{1}{v}\right)=\frac{\mu}{v}\tilde{u}_y - q(v),
        \end{equation}
       We start by using the transport theorem to get 
           \begin{equation*}
               \partial_t \left( \frac{1}{2} \int_{-\eta_{out}(t)}^0 \tilde{u}^2(t,y) \, {\rm d}y \right)= \frac{1}{2}\left(\int_{-\eta_{out}(t)}^0 (\tilde{u}^2)_t(t,y) + \partial_{y}( \tilde{u}^2 w )(t,y)\, {\rm d}y\right)
           \end{equation*}
     We integrate by parts with respect to the space variable and we take into account \eqref{ss} to obtain 
           \begin{multline*}
              \frac{1}{2}\left( \int_{-\eta_{out}(t)}^0 (\tilde{u}^2)_t(t,y)\, {\rm d}y \right)= \int_{-\eta_{out}}^0 \tilde{u} (t,y) \sigma_y\, {\rm d}y,\\
               =-\int_{-\eta_{out}(t)}^0 \tilde{u}_y(t,y) \sigma\, {\rm d}y +u(t,0)\sigma(t,0)-\tilde{u}(t,-\eta_{out}(t))\sigma(t,-\eta_{out}(t)).
           \end{multline*}
By
\begin{equation*}
    \sigma(t,0)= -\ddot{b}(t)-k(b(t)-b_b)-l\dot{b}(t)
\end{equation*}
we obtain
\begin{multline}\label{002}
   \frac{1}{2} \left(\int_{-\eta_{out}(t)}^0 (\tilde{u}^2)_t(t,y)\, {\rm d}y\right) = -\int_{-\eta_{out}(t)}^0 \tilde{u}_y \sigma (t,y)\, {\rm d}y - \frac{1}{2}\frac{d}{dt} \left[|\dot{b}(t)|^2-k|b(t)-b_b|^2\right]- l\dot{b}^2(t) \\
   -\tilde{u}(t,-\eta_{out}(t)) \sigma (t,-\eta_{out}(t)).
\end{multline}
On the other hand, we have 
\begin{equation*}
    \frac{1}{2}\int_{-\eta_{out}(t)}^0 (\tilde{u}^2(t,y) w(t,y))_y\, {\rm d}y = \frac{1}{2} \left( \tilde{u}^2(t,0) w(t,0) - \tilde{u}^2(t,-\eta_{out}(t)) w(t,-\eta_{out}(t))\right),
\end{equation*}
and thus
\begin{equation}\label{001}
   \frac{1}{2}\left(\int_{-\eta_{out}(t)}^0 (\tilde{u}^2(t,y) w(t,y))_y\, {\rm d}y\right)= \frac{\tilde{u}_{out}^3(t)}{2v(t,-\eta_{out}(t))}, 
\end{equation}
We use \eqref{001} and \eqref{002} to get
\begin{multline*}
   \frac{1}{2} \partial_t\left[\int_0^{b(t)} (\tilde{u}^2(t,y))\, {\rm d}y + |\dot{b}(t)|^2+ k|b(t)-b_b|^2\right] +l\dot{b}(t)^2\\
  = \frac{\tilde{u}_{out}^3(t)}{2v(t,\eta_{out}(t)}-\int_{-\eta_{out}(t)} \tilde{u}_y \sigma(t,y)\, {\rm d}y - \tilde{u}(t,-\eta_{out}(t)) \sigma(t,-\eta_{out}(t)).
\end{multline*}
This yields
\begin{multline}\label{012}
    \frac{1}{2} \partial_t\left[\int_0^{b(t)} (\tilde{u}^2(t,y))\, {\rm d}y + |\dot{b}(t)|^2+ k|b(t)-b_b|^2\right] +l\dot{b}(t)^2\\ 
    =\frac{u_{out}^3(t)}{2v(t,\eta_{out}(t)}-\int_{-\eta_{out}(t)}^0 \frac{\mu}{v(t,y)}(\tilde{u}_y(t,y))^2\, {\rm d}y+ \int_{-\eta_{out}(t)}^0 q(v) \tilde{u}_y(t,y)\, {\rm d}x  \\ - \tilde{u}(t,-\eta_{out}(t)) \sigma(t,-\eta_{out}(t)). 
\end{multline}
Recall $Q$ is the primitive of $q$ with respect to $v$ thus we have 
\begin{equation}\label{011}
    \int_{-\eta_{out}(t)}^0 q(v) \tilde{u}_y\, {\rm d}y = \frac{d}{dt} \int_{\eta_{out}(t)}^0 Q(v)\, {\rm d}y - \frac{Q(v)}{v}(t,-\eta_{out}(t)).\tilde{u}_{out}(t).
\end{equation}
We combine \eqref{012}  and \eqref{011}, we obtain 
\begin{multline} \label{45}
     \frac{1}{2} \partial_t\left[\int_0^{b(t)} (\tilde{u}^2(t,y)) -Q(v)\, {\rm d}x + |\dot{b}(t)|^2+ k|b(t)-b_b|^2\right] +l\dot{b}(t)^2 \\ =\frac{\tilde{u}_{out}^3(t)}{2v(t,\eta_{out}(t))} - \tilde{u}(t,-\eta_{out}(t)) \sigma(t,-\eta_{out}(t))-  \frac{Q(v)}{v}(t,-\eta_{out}(t)).\tilde{u}_{out}(t).
\end{multline}
We multiply now, the equation \eqref{ss} by $-\tilde{u}_{out}(t)$
\begin{equation*}
    \int_{-\eta_{out}(t)}^0 \tilde{u}_t \tilde{u}_{out}(t)\, {\rm d}y = \int_{-\eta_{out}(t)}^0 \sigma_y \tilde{u}_{out}(t)\, {\rm d}y,  
\end{equation*}
we use the fact that
\begin{multline*}
 \tilde{u}_{out}(t) \int_{-\eta_{out}(t)}^0 \tilde{u}_t(t,y) \, {\rm d}x=  \tilde{u}_{out}(t)\left(\partial_t \left(\int_{\eta_{out}(t)}^0 \tilde{u} (t,y)\, {\rm d}y \right)  - \int_{\eta_{out}(t)}^0(\tilde{u}.w)_y(t,y)\, {\rm d}y \right ) \\
  = \tilde{u}_{out}(t)\left(\partial_t \left(\int_{\eta_{out}(t)}^0 \tilde{u} (t,y)\, {\rm d}y \right) -\tilde{u}(t,0)w(t,0)+ \tilde{u}(t,-\eta_{out}(t))w(t,-\eta_{out}(t)) \right) \\ = \tilde{u}_{out}(t) \partial_t \left(\int_{\eta_{out}(t)}^0 \tilde{u} (t,y)\, {\rm d}y \right) -\frac{\tilde{u}_{out}(t)^3}{v(t,-\eta_{out}(t))},
\end{multline*}
on the other hand, we have 
\begin{equation*}
    \tilde{u}_{out}(t) \int_{-\eta_{out}(t)}^0 \sigma_y(t,y)\, {\rm d}y = \tilde{u}_{out}(t)(\sigma(t,0)-\sigma(t,-\eta_{out}(t)))
\end{equation*}
and consequently, we obtain 
\begin{equation}\label{44}
\tilde{u}_{out}(t) \frac{d}{dt} \left(\int_{\eta_{out}(t)}^0 \tilde{u} (t,y)\, {\rm d}y \right) -\frac{u_{out}(t)^3}{v(t,-\eta_{out}(t))}= \tilde{u}_{out}(t)(\sigma(t,0)-\sigma(t,-\eta_{out}(t)))
\end{equation}
At this level we add \eqref{45} and \eqref{44}, we obtain 
\begin{multline*}
    \frac{1}{2}\partial_t \left[ \int_{-\eta_{out}(t)}^0 \left(\tilde{u}^2(t,y)-Q(v)\right)\, {\rm d}y +|\dot{b}(t)|^2 +k|b(t)-b_b|^2\right] +l(\dot{b}(t))^2 +\frac{Q(v)}{v}(t,-\eta_{out}(t))\tilde{u}_{out}(t) \\ +\int_{-\eta_{out}(t)}^0 \frac{\mu}{v(t,x)} |\tilde{u}_y(t,y)|^2\, {\rm d}y= \tilde{u}_{out}(t) \left( \partial_t \int_{-\eta_{out}(t)}^0 \tilde{u}(t,y)\, {\rm d}y \right)- \frac{\tilde{u}_{out}^3(t)}{2v(t,-\eta_{out}(t))} \\ + \tilde{u}_{out}(t)\left[\ddot{b}(t)+k(b(t)-b_b)+l\dot{b}(t)\right]
\end{multline*}
We integrate all the terms in the above formula with respect to time and we get
\begin{multline*}
    \frac{1}{2}\left(\int_{-\eta_{out}(t)}^0 \left(\tilde{u}^2(t,y)-Q(v)\, {\rm d}y\right)+|\dot{b}(t)|^2 + K|b(t)-b_b|^2\right) \\ +l \int_{T^{\star}}^t \dot{b}^2(s)\, {\rm d}s +\int_{0}^t \int_{-\eta_{out}(t)}^0 \frac{\mu}{v} |\tilde{u}_y(t,y)|^2\, {\rm d}y +\int_{0}^t \frac{Q(v)}{v} \tilde{u}_{out}(s)ds \\= -\int_{T^{\star}}^t \frac{\tilde{u}_{out}^3(s)}{2v(s,-\eta_{out}(s))}\, {\rm d}s +\left[\tilde{u}_{out}(s) \dot{b}(s)\right]^{s=t}_{t=0}+\tilde{u}_{out}(0)b_b-  \tilde{u}_{out}(t)b(t) \\ +\int_{T^{\star}}^t \tilde{u}_{out}(s) (k(b(s)-b_b)-l\dot{b}(s))\, {\rm d}s +\left[\tilde{u}_{out}(t)(s) \int_{-\eta_{out}(s)}^ 0 \tilde{u}(s,y)dy\right]^{s=t}_{s=0} \\ -\int_{0}^t \partial_s\tilde{ u}_{out}(s) \int_{-\eta_{out}(t)}^0 \tilde{u}(s,y)\, {\rm d}y\, {\rm d}s
\end{multline*}
which may be rewritten as
\begin{multline*}
       \frac{1}{2}\left(\int_{-\eta_{out}(t)}^0 \left(\tilde{u}^2(t,y)-Q(v)\, {\rm d}y\right)+|\dot{b}(t)|^2 + K|b(t)-b_b|^2\right)  \\ +l \int_{T^{\star}}^t \dot{b}^2(s)\, {\rm d}s +\int_{0}^t \int_{-\eta_{out}(t)}^0 \frac{\mu}{v} |\tilde{u}_y(t,y)|^2\, {\rm d}y +\int_{0}^t \frac{-\tilde{u}_{out}(s)}{\gamma-1} \frac{1}{v^{\gamma}}ds  \\= -\int_{T^{\star}}^t \frac{\tilde{u}_{out}^3(s)}{2v(s,-\eta_{out}(s))}\, {\rm d}s +\left[\tilde{u}_{out}(s) \dot{b}(s)\right]^{s=t}_{s=T^*}+\tilde{u}_{out}(0)b_b-  \tilde{u}_{out}(t)b(t) \\ +\int_{T^{\star}}^t \tilde{u}_{out}(s) (k(b(s)-b_b)-l\dot{b}(s))\, {\rm d}s +\left[\tilde{u}_{out}(t)(s) \int_{-\eta_{out}(s)}^ 0 \tilde{u}(s,y)dy\right]^{s=t}_{s=0} \\ -\int_{T^*}^t \partial_s \tilde{u}_{out}(s) \int_{-\eta_{out}(t)}^0 \tilde{u}(s,y)\, {\rm d}y\, {\rm d}s
\end{multline*}
    We next estimate the terms in the right side using the Young inequality, we obtain the following estimations 
    \begin{multline}\label{41}
        \left[\tilde{u}_{out}(s)\dot{b}(s)\right]^{s=t}_{s=T^{\star}} \leq  4 |\tilde{u}_{out}(t)|^2+ \frac{|\dot{b}(t)|^2}{4}- \tilde{u}_{out}(T^*)\dot{b}(T^*)\\
        \leq C \left(\|\tilde{u}_{out}\|^2_{L^{\infty}(T^{\star},T_1)}+1\right)+ \frac{|\dot{b}(t)|^2}{4},
    \end{multline}
    \begin{multline}\label{5}
        \left|\int_{T^*}^t \tilde{u}_{out}(s) (k(b(s)-b_b)-l\dot{b}(s))\, {\rm d}s \right|\\ \leq C\|\tilde{u}_{out}(t)\|^2_{L^2(T^{\star},T_1)}+k^2 \int_{T^*}^t|b(s)-b_b|^2\, {\rm d}s +\frac{l}{2}\int_{T^{\star}}^t \dot{b}(s)^2 \, {\rm d}s,
    \end{multline}
    \begin{equation}\label{6}
        \left[\tilde{u}_{out}(t)(s) \int_{-\eta_{out}(s)}^ 0 \tilde{u}(s,y)\, {\rm d}y\right]^{s=t}_{s=T^*} \leq C \|u_{out}\|^2_{L^{\infty}(T^{\star},T_1)}+\frac{1}{8} \int_{-\eta_{out}(t)}^0 \tilde{u}^2(s,y)\, {\rm d}y,
    \end{equation}
    and
    \begin{equation}\label{7}
      \left | -\int_{T^{\star}}^t \partial_{s} \tilde{u}_{out}(s) \int_{-\eta_{out}(s)}^0 \tilde{u}(s,y)\, {\rm d}y\, {\rm d}s \right| \leq \int_{T^*}^t\left( |\partial_t \tilde{u}_{out}(s)|^2+\int_{-\eta_{out}(t)}^0 \tilde{u}^2(s,y)\, {\rm d}y \right)\, {\rm d}s 
    \end{equation}
  We gather \eqref{41}, \eqref{5} \eqref{6} and \eqref{7} and we conclude that there exists $C$ dependent on initial conditions such that for every $t\in [T^{*}, T_1]$ we have 
    \begin{multline*}
\left(\int_{-\eta_{out}(t)}^0 \left(\frac{3}{8}\tilde{u}^2(t,y)-Q(v)\, {\rm d}y\right)+\frac{1}{4}|\dot{b}(t)|^2 + \frac{k}{2}|b(t)-b_b|^2\right)  \\ +\frac{l}{2} \int_{T^{\star}}^t \dot{b}^2(s)\, {\rm d}s +\int_{T^{*}}^t \int_{-\eta_{out}(t)}^0 \frac{\mu}{v} |\tilde{u}_y(t,y)|^2\, {\rm d}y +\int_{T^{\star}}^t \frac{-u_{out}(s)}{\gamma-1} \frac{1}{v^{\gamma}}ds\\ \leq C\left( 1+\int_{T^{\star}}^t |\tilde{u}_{out}(s)|^{\frac{3\gamma-1}{\gamma-1}}\, {\rm d}s +\|\tilde{u}_{out}\|_{L^{\infty}(T^{\star},T)}^2+\|\tilde{u}_{out}\|^2_{H^1(T^{\star},T_1)}+\int_{T^{\star}}^t |\dot{b}(s)|^2\, {\rm d}s\right)\\+\int_{T^{\star}}^t \int_{-\eta_{out}(t)}^0 \tilde{u}^2(s,y)\, {\rm d}y\, {\rm d}s +k^2\int_{T^*}^t |b(s)-b_b|^2\, {\rm d}s , 
    \end{multline*}
Finally, using Gronwall's Lemma, we deduce \eqref{EE} holds for $\gamma>1$.
       \end{proof}
\begin{remark}
All the following estimates are deduced under assumption $b(t)>0$. This can be ensured by the energy inequality and by additional assumption on the smallness of $T_0$. The precise estimate on the time of existence is given later in Section \ref{con.rem}.
\end{remark}
\begin{lemma}
We have \label{lower.bound.v}
$$
\left\|\frac 1 v\right\|_{L^\infty((T^*,T_2)\times (-\eta_{out}(t),0))}\leq C
$$
\end{lemma}
\begin{proof}
We set $M(v)$ to be a potential of $\frac \mu v$, i.e. $M(v) = \mu \log v$ . The momentum equation takes the form
$$
\tilde u_t + (q(v))_y = (M(v))_{ty}.
$$
We integrate this equation with respect to time and space over a set $(T^*,t_0)\times (y_0,0)$ where $t_0\in (T^*,T^1]$ and $y_0\in (-\eta_{out},0)$ are arbitrary. We arrive at (below we skip the dependence on $\tilde u$ and $v$ to keep the notation concise)
\begin{multline*}
\int_{y_0}^0 \tilde u(t_0,y)\, {\rm d}dy - \int_{y_0}^0 \tilde u(T^*,y)\, {\rm d}y + \int_{T^*}^t q(t,0)\, {\rm d}t -  \int_{T^*}^t q(t,y_0)\, {\rm d}t\\ = M(t_0,0) - M(t_0,y_0) - M(T^*,0) + M(T^*,y_0)
\end{multline*}
and we rearrange it as 
\begin{multline}\label{bound.on.M}
 -M(t_0,y_0) = -M(t_0,0) +  M(T^*,0) - M(T^*,y_0)\\ + \int_{T^*}^{t_0} q(t,0)\, {\rm d}t - \int_{T^*}^{t_0}q(t,y_0)\, {\rm d}t + \int_{y_0}^0 u(T^*,y) - u(t_0,y)\, {\rm d}y
\end{multline}
The boundary condition of \eqref{0101} yields 
$$
q(t,0) = M_t(t,0) + \ddot b(t) + l\dot b(t) + K(b(t) - b_0))
$$
and thus
$$
\int_{T^*}^{t_0} q(t,0)\, {\rm d}t = M(t_0,0) - M(T^*,0) + |\dot b(t_0)| - |\dot b(T^*)|  + l|b(t_0)| - l|b(T^*)| + \int_{T^*}^{t_0}K(b(t) - b_0)\, {\rm d}t.
$$
Since $q$ is positive by definition, the sum of the right hand side of \eqref{bound.on.M} might be estimated from above by energy and thus
$$
M(t_0,y_0)\geq -C>-\infty
$$
and the definition of $M$ yields the demanded claim.
\end{proof}
\begin{lemma} \label{lemme.v}
We have $v\in L^\infty((T^*,T_2,H^1(-\eta_{out},0))$ and, consequently, $v\in L^\infty((T^*,T^2)\times (-\eta_{out}(t),0))$. 
\end{lemma}
\begin{proof}
We use the potential $M$ defined in the previous proof. We have $(M(v))_t = \frac \mu v v_t  = \frac \mu v \tilde u_y$. The momentum equation can be thus rewritten as
\begin{equation}\label{v.potential}
(\tilde u - (M(v))_y)_t = -(q(v))_y
\end{equation}
We use the transport theorem and the function $w$ to compute as follows
\begin{multline*}
\partial_t\frac 12\int_{-\eta_{out}}^0 |\tilde u - (M(v))_y|^2\, {\rm d}y\\
= \int_{-\eta_{out}}^0 (\tilde u - (M(v))_y)_t(\tilde u - (M(v))_y)\, {\rm d}y + \int_{-\eta_{out}}^0\left(\frac 12|\tilde u - (M(v))_y|^2 w\right)_y\, {\rm d}y\\
= - \int_{-\eta_{out}}^0 (q(v))_y (\tilde u - (M(v))_y)\, {\rm d}y - \left(\frac 12|\tilde u - (M(v))_y|^2 w\right)(-\eta_{out})
\end{multline*}
We have $$(q(v))_y = -\gamma \frac 1{v^{\gamma+1}}v_y = -\frac\gamma{\mu v^\gamma} \frac \mu v v_y = -\frac \gamma{\mu v^\gamma}(M(v))_y$$ by definition and, therefore, \eqref{v.potential} yields (recall $w(-\eta_{out})$ is positive)
\begin{multline*}
\frac 12\partial_t\int_{-\eta_{out}}^0 |\tilde u - (M(v))_y|^2\, {\rm d}y\leq \frac \gamma \mu\int_{-\eta_{out}}^0 \frac 1{v^\gamma} (M(v))_y(\tilde u - (M(v))_y)\, {\rm d}y \\
\leq -\frac \gamma \mu\int_{-\eta_{out}}^0 \frac 1{v^\gamma} (\tilde u - M(v))(\tilde u - (M(v))_y)\, {\rm d}y + -\frac \gamma \mu\int_{-\eta_{out}}^0 \frac 1{v^\gamma} \tilde u(\tilde u - (M(v))_y)\, {\rm d}y.
\end{multline*}
Since $\frac 1 v$ is in $L^\infty((T^*,T_2)\times (-\eta_{out}(t),0))$ and $\tilde u\in L^\infty(T^*,T_2,L^2(-\eta_{out}(t),0))$, we deduce
$$
\partial_t\int_{-\eta_{out}}^0 |\tilde u - (M(v))_y|^2\, {\rm d}y\leq C \left(1+ \int_{-\eta_{out}}^0 |\tilde u - (M(v))_y|^2\, {\rm d}y \right)
$$
and the Gronwall inequality yields the demanded claim.
\end{proof}
       \begin{lemma}
           There exists a constant $C>0 $
          such that $$\|\tilde{u}_y(t,.)\|^2_{L^2(-\eta_{out}(t),0)}+\int_{0}^t \|\tilde{u}_{yy}(s,.)\|^2_{L^2(-\eta_{out}(t),0)}\, {\rm d}s  \leq C, \ \ \forall t\in [T^{\star},T_2],$$
       \end{lemma}
        \begin{proof}
           We multiply \eqref{0101}${_2}$ by $-\tilde{u}_{yy}$ and integrate with respect to space interval $[-\eta_{out}(t),0]$, we obtain 
           \begin{multline}\label{1997}
               -\int_{-\eta_{out}(t)}^0 \tilde{u}_t u_{yy}\, {\rm d}y+\int_{-\eta_{out}(t)}^0 \left(\frac{\mu }{v}\right)|\tilde{u}_{yy}|^2  dy=-\int_{-\eta_{out}(t)}^0 \mu \left(\frac{1}{v^2}\right) v_y \tilde{u}_y \tilde{u}_{yy} dy+  \int_{-\eta_{out}(t)}^0 (q(v))_y u_{yy}{\rm d}y 
         \end{multline}
        Integration by parts with respect to space yields 
        \begin{multline}\label{214}
- \int_{-\eta_{out}(t)}^0 \tilde{u}_t u_{yy}\, {\rm d}y
      =\frac{1}{2}\int_{-\eta_{out}(t)}^0 (|\tilde{u}_{y}|^2)_t  \, {\rm d}y + \tilde{u}_t(t,-\eta_{out}(t)).\tilde{u}_y(t,-\eta_{out}(t))-     \tilde{u}_t(t,0)\tilde{u}_y(t,0).
         \end{multline}
        By the boundary conditions
        \begin{align*}
            \tilde{u}_t(t,0)=\ddot{b}(t), \ \text{and} \ \tilde{u}(t,\eta_{out}(t))=\tilde{u}_{out}(t), 
        \end{align*}
         the formula \eqref{1997} becomes 
         \begin{multline}
        \frac{1}{2}\int_{-\eta_{out}(t)}^0 (|\tilde{u}_{y}|^2)_t  \, {\rm d}y  +\left(\tilde{u}_{out}(t)\right)_t\tilde{u}_y(t,-\eta_{out}(t))-     \left(\frac{d}{dt} \dot{b}(t)\right)\tilde{u}_y(t,0)  +\int_{-\eta_{out}(t)}^0 \left(\frac{\mu }{v}\right)|\tilde{u}_{yy}|^2   \, {\rm d}y  \\=-\int_{-\eta_{out}(t)}^0 \mu \left(\frac{1}{v^2}\right) v_y \tilde{u}_y \tilde{u}_{yy} dy+  \int_{-\eta_{out}(t)}^0 (q(v))_y \tilde{u}_{yy} {\rm d}y.
         \end{multline}
         now, using the energy estimate \eqref{EE}, the result of Lemma \ref{lower.bound.v} and the Young inequality, we deduce that for an arbitrary constant $\epsilon$ and $c(\epsilon)$, we have          \begin{align} \label{099}
            \ddot{b} (t)\tilde{u}_y(t,0) =\left(l\dot{b}(t)+k(b(t)-b_b)+\mu\frac{\tilde{u}_y}{v}(t,0)-q(v(t,0))\right)\tilde{u}_y(t,0) \nonumber\\ \leq c(\epsilon)+ c(\epsilon,\mu)\|\tilde{u}_y\|^2_{L^{\infty}(-\eta_{out(t)},0)},
       \end{align}
         and 
         \begin{equation} \label{09}
            (\tilde{u}_{out})_t\tilde{u}_y(t,-\eta_{out}(t)) \leq c(\epsilon) + \epsilon  \|\tilde{u}_y\|^2_{L^{\infty}(-\eta_{out(t)},0)}.
         \end{equation}
         On the other hand, the Sobolev embedding theorem yields
         \begin{equation}\label{0009}
             \|\tilde{u}_y\|_{L^{\infty}(-\eta_{out}(t)),0)} \leq C \|\tilde{u}_{yy}\|_{L^2(-\eta_{out}(t),0)}+ C \|\tilde{u}_y\|_{L^2(-\eta_{out}(t),0)}.
         \end{equation}
        We comeback to the formula \eqref{1997},  and taking into account \eqref{099}, \eqref{09} and \eqref{0009}, we get 
         \begin{multline}\label{157}
             \frac{1}{2}\int_{-\eta_{out}(t)}^0 (|\tilde{u}_{y}|^2)_t  \, {\rm d}y +c(\epsilon)+c(\epsilon,\mu) \|\tilde{u}_y\|_{L^{\infty}(-\eta_{out},0)}  + \left\| \frac{\mu}{v}\right\|_{L^{\infty}(-\eta_{out},0) }\|\Tilde{u}_{yy}\|_{L^2(-\eta_{out}(t),0)} \\
            + \left\|\frac{\mu}{v^2}(t,.)\right\|_{L^{\infty}(-\eta_{out},0)}\left( c(\epsilon) \int_{-\eta_{out}(t)}^0 |v_y|^2|\Tilde{u}_y|^2 {\rm d}y+\epsilon\|\Tilde{u}_{yy}\|^2_{L^2(-\eta_{out}(t),0)}  \right) \\+\left\| \frac{\gamma}{v^{\gamma+1}}\right\|_{L^{\infty}(-\eta_{out},0)}\left(c(\epsilon) \int_{-\eta_{out}(t),0} |v_y|^2 {\rm d}y +\epsilon \|\Tilde{u}_{yy}\|^2_{L^2(-\eta_{out},0)}\right)  
         \end{multline}
Due to Lemma \ref{lower.bound.v} we get
          \begin{multline}\label{157}
             \frac{1}{2}\int_{-\eta_{out}(t)}^0 (|\tilde{u}_{y}|^2)_t  \, {\rm d}y +
             \int_{-\eta_{out}(t)}^0 |\Tilde{u}_{yy}|^2 \leq \epsilon \|\Tilde{u}_{yy}\|^2_{L^2(-\eta_{out})(t),0)} \\
            + c(\epsilon) \left(\|\Tilde{u}_y\|^2_{L^2(-\eta_{out}(t),0)}+\|\Tilde{u}\|^2_{L^2(-\eta_{out}(t),0)} + \int_{-\eta_{out}(t)}^0 |v_y|^2|\Tilde{u}_y|^2 {\rm d}y  +\|v_y\|^2_{L^2(-\eta_{out}(t),0)}\right).
         \end{multline}
Now, we integrate with respect to time over $[T^{\star},t]$ for all, $t\in [T^{\star},T_2]$ and using the Lemma \ref{lemme.v} and Gronwall's inequality, we obtain the desired estimate.
       \end{proof}
Due to the previous bounds, we are able to apply Theorem \ref{outflow.local} with initial time $T_2$ and we may extend the solution beyond the interval. As a result of this ideas, we state the following theorem

\begin{theorem}\label{main.thm}
The strong solution to \eqref{11} exists on a time interval $[T^*,T]$ for every $T>T^*$ assuming $b(t)>0$ for every $t\in [0,T]$.
\end{theorem}

    \section{Concluding remarks}\label{con.rem}
Below, we give an estimate on the maximal time of existence. As pointed in Theorem \ref{main.thm}, the solution exists as far as $b(t)>0$. To find an estimate, we first specify the estimate given in Lemma \ref{lower.bound.v}. The definition of $M$ together with \eqref{bound.on.M} yields
\begin{multline}\label{lower.precise}
\frac 1{v(t,y)}\leq\frac 1{v(T^*,y)}\cdot \\
\cdot\exp \left(\frac 1\mu\left(\dot b(t) - \dot b(T) + l b(t) - l b(T) + \int_{T^*}^t K(b(s) - b_0)\, {\rm d}s + \sqrt{\eta_{out}}\left(\|\tilde u(t,\cdot)\|_2 + \|\tilde u(T^*,\cdot)\|_2\right)\right)\right)
\end{multline}
We denote the exponent in \eqref{lower.precise} as $G(t)$ so \eqref{lower.precise} can be written easily
$$\frac 1{v(t,y)}\leq \frac 1{v(T^*,y)} e^{G(t)}.$$
According to Lemma \ref{specific.vol} we have (recall $\eta_{out}<0$)
\begin{multline*}
b(t) = \int_{-\eta_{out}}^0 v(t,y)\, {\rm d}y \geq  e^{-G(t)} \eta_{out}(t) \min_{y\in [-\eta_{out}(T^*), 0]} v(T^*,y)\\
\geq C \left[\eta_{out}(T^*) + \int_0^t u_{out}(\tau)\, {\rm d}\tau\left(\min_{\tau} v(\tau,-\eta_{out})\right)^{-1}\right]\geq C\left[\eta_{out}(T^*) + C\int_0^t u_{out}(\tau)\, {\rm d}\tau\right]
\end{multline*}
where $C$ depends on energy. Thus we give the lower estimate on the maximal time of existence -- it is such time $T_3$ that 
$$
C\int_{T^*}^{T_3} u_{out}(t)\, {\rm d}t = -\eta_{out}(T^*)
$$
where $C$ depends only on initial and boundary conditions.\\

\section*{Acknowledgment}
All authors have been supported by Praemium Academiæ of \v S. Ne\v casov\' a. S. Chebbi (first version) was supported by  by the Czech Science Foundation (GA\v CR) through project 19-04243S.  Further, V. M\'acha \and \v S. Ne\v casov\'a 
would like to announce the support from project GA22-01591S in the framework of RVO:67985840.

\end{document}